\documentclass{article}
\usepackage{amssymb,amsmath}
\usepackage{amsthm}
\usepackage{hyperref}
\usepackage{amscd}
\newtheorem{theorem}{Theorem}
\newtheorem{proposition}[theorem]{Proposition}

\theoremstyle{definition}

\newtheorem*{remark}{Remark}
\DeclareMathOperator{\Gal}{\mathrm{Gal}}
\DeclareMathOperator{\GalQ}{\Gal(\overline{\mathbf{Q}}/\mathbf{Q})}
\DeclareMathOperator{\GL}{\mathrm{GL}}
\DeclareMathOperator{\GU}{\mathrm{GU}}
\DeclareMathOperator{\PGL}{\mathrm{PGL}}
\DeclareMathOperator{\PGU}{\mathrm{PGU}}
\DeclareMathOperator{\PU}{\mathrm{PU}}
\DeclareMathOperator{\PSL}{\mathrm{PSL}}

\DeclareMathOperator{\PSU}{\mathrm{PSU}}

\DeclareMathOperator{\Qbar}{\overline{\mathbf{Q}}}

\DeclareMathOperator{\et}{\mathrm{\acute{e}t}}
\DeclareMathOperator{\SlU}{\mathrm{S}_{\mathit{l}}\mathrm{U}}
\DeclareMathOperator{\SlL}{\mathrm{S}_{\mathit{l}}\mathrm{L}}
\begin{document}
	\author{Stepan Nesterov}
	\title{Simple Lie Groups of type $A_n$ as Galois groups over $\mathbf{Q}$}
	\maketitle
	\begin{abstract}
		In this paper, we utilize our previous results on mod $p$ monodromy of cyclic coverings of the projective line to realize a large series of groups of the form $\PSL(n,q)$ and $\PSU(n,q)$ as Galois groups over $\mathbf{Q}$. We achieve for the first time a fully explicit infinite series of such groups where simultaneously the field can have arbitrarily large degree over the prime field and the group does not coincide with $\PGL(n,q)$ or $\PGU(n,q)$, respectively. 
	\end{abstract}
	\section{Introduction}
	The Inverse Galois Problem is to find, for a given finite group $G$, a Galois extension of $\mathbf{Q}$ with Galois group $G$. We use \cite{malle} as a general reference regarding the Inverse Galois Problem. In view of the classification of finite simple groups, it is natural to attempt to solve the Inverse Galois Problem for the simple groups frirst. Despite a lot of progress, for many of these groups the Inverse Galois problem remains open. \par 
	To solve the Inverse Galois Problem for the alternating groups $A_n$ is not that difficult, and was already done by Hilbert. One can even give explicit examples: indeed, for any $n$, the splitting field of a polynomial $x^n - x - 1$ has Galois group $S_n$; the polynomial $1+x+\frac{x^2}{2} + \ldots \frac{x^n}{n!}$ has Galois group $S_n$ if $n$ is not a multiple of $4$, and $A_n$ otherwise. See \cite{alternating} for a large family of explicit polynomials with Galois group $A_n$. \par 
	One of the most productive methods to solve the Inverse Galois Problem is the rigidity method, originally presented in \cite{rigidity}. In addition to recovering the result for $S_n$ or $A_n$, it allows us to solve the Inverse Galois Problem for a lot of simple groups of Lie type over prime fields $\mathbf{F}_p$, and all the sporadic groups with the exception of $M_{23}$ and $M_{24}$. See \cite[Chapter II]{malle} for an extensive list of groups where the rigidity method applies. \par 
	This leaves us with a large class of simple groups of Lie type over finite fields $\mathbf{F}_q$, where $q$ is not a prime, but a prime power. In \cite[Chapter III.10]{malle} a generalisation of the rigidity method is discussed, which allows us to cover many simple groups associated to symplectic and orthogonal groups over finite fields. An alternative approach is to prove that Galois representations attached to automorphic forms often have full image \cite{khare} \cite{wiese}. However these results are ineffective in the sense that it you do not know which field $\mathbf{F}_{p^k}$ you get as a field of coefficients, only that $k$ can be arbitrality large. \par 
	In this paper we give, for the first time, a realisation of an infinite set of simple groups of the form $\PSL(n, q)$ and $\PSU(n, q)$ as Galois groups over $\mathbf{Q}$, where simultaneously $n$ can be arbitrarily large, $\mathbf{F}_q$ can have arbitraily large degree over the prime field, and the condition on $(n,q)$ to arise as a Galois group over $\mathbf{Q}$ is complicated, but completely explicit. We achieve this by considering the action of the absolute Galois group $G_{\mathbf{Q}}$ on the etale cohomology with finite coefficients of cyclic coverings of $\mathbf{P}^1$. The deck transformation of the curve coming from the covering map to $\mathbf{P}^1$ is what endows the cohomology with the structure of an $\mathbf{F}_q$-vector space. \par 
	The main result of this paper is: 
	\begin{theorem} \label{final}
		Let $l$ be an odd prime, and let $n$ be a natural number, not congruent to $-2$ modulo $l$. Let $e = \mathrm{ord} \ p \mod l$. Then:
		\begin{itemize}
			\item If $\frac{l-1}{e}$ is even, then $\PSL((l-1)n-2,p^e)$ is a Galois group over $\mathbf{Q}$.
			\item If $\frac{l-1}{e}$ is odd and, in addition to all the previous assumptions, either $e$ is a multiple of $4$ or if $2^r$ is the largest power of $2$ dividing $p+1$, then $\frac{l-1}{2}n \equiv 1 \mod 2^r$, then $\PSU((l-1)n - 2, p^{\frac{e}{2}})$ is a Galois group over $\mathbf{Q}$.
		\end{itemize}
	\end{theorem}
	For various choices of $l$ and $p$, we have the following corollaries:
	\begin{itemize}
		\item For a prime $p = 3k + 1$ and $n$ not divisible by $3$, $\PSL(2n,p)$ is a Galois group over $\mathbf{Q}$. Under these congruence conditions, the corresponding groups $\PGL(2n,p)$ are rigid, so under the additional assumption $\mathrm{gcd}(n, p - 1) \le 2$, the simple groups $\PSL(n, p-1)$ were previously obtained as Galois groups over $\mathbf{Q}$ by the rigidity method \cite[Theorem III.10.3]{malle}. In the case $n=2$, the Inverse Galois Problem is solved in \cite{zywina}.
		\item For a prime $p = 5k-1$, and $n$ odd, but not divisible by $5$, $\PSL(2n, p^2)$ is a Galois group over $\mathbf{Q}$. In particular, $\PSL(2, p^2)$ is a Galois group over $\mathbf{Q}$ for such $p$. This complements a result of Mestre \cite{mestre}, where $\PSL(2, p^2)$ is realised as a Galois group over $\mathbf{Q}$ for $p = 5k \pm 2$.
		\item For a prime $p = 5k + 1$, and $n$ odd but not divisible by $5$, $\PSL(2n, p)$ is a Galois group over $\mathbf{Q}$.
		\item If $p$ and $l$ are two odd primes, and $p$ has multiplicative order $\frac{l-1}{2}$ mod $l$, then $\PSL(l-3, p^{\frac{l-1}{2}})$ is a Galois group over $\mathbf{Q}$. So this method allows us to obtain groups $\PSL(n,q)$ where $q$ is an arbitrarily large power of a prime, but the rank of the group will also become arbitrarily large.
		\item For a prime $p = 12k +7$ and $n$ not divisible by $3$, $\PSU(2n ,p)$ is a Galois group over $\mathbf{Q}$. Again, these congruence conditions also arise when applying the rigidity method.  \cite[Theorem III.10.3]{malle}
		\item For a prime $p=28k+13$ and $n$ odd, but not congruent to $-2$ mod $7$, $\PSU(6n-2 ,p)$ is a Galois group over $\mathbf{Q}$. 
		\item For a prime $p = 28k+5, 28k+17$, and $n$ odd, but not congruent to $-2$ mod $7$, $\PSU(6n-2 ,p^3)$ is a Galois group over $\mathbf{Q}$.
		\item If $p$ and $l$ are two odd primes, such that $l = 4m+1$, and $p$ is a primitive root mod $l$, then $\PSU(l-3, p^{\frac{l-1}{2}})$ is a Galois group over $\mathbf{Q}$.
	\end{itemize}
	\section{Proof Outline}
	In a previous article, we studied the monodromy action on cohomology of cyclic coverings of $\mathbf{P}^1$ and proved the following theorem:	
	\begin{theorem}
		If $n \geq l$, then the image of the monodromy representation of the pure braid group on the cohomology of cyclic covering of $\mathbf{P}^1$ branched at $n$ points is either the group $\SlU(n-1,q) := \{A \in \GL(n-1,q^2) : A \cdot \mathrm{Fr}_q(A)^t = 1, \det A \in \mu_l \}$ or $\SlL(n-1,q) := \{A \in \GL(n-1,q) : \det A \in \mu_l\}$, depending on whether or not $\frac{l-1}{\mathrm{ord (p \mod l)}}$ is odd or even.
	\end{theorem}
	The naive model of a cyclic covering of $\mathbf{P}^1$ with monodromy vector $(k_1, \ldots, k_n)$ branched at $(\lambda_1, \lambda_2, \ldots, \lambda_{n-1}, \infty)$ is the normalization of the projective closure of the affine curve $y ^ l = (x - \lambda_1)^{k_1} \ldots (x - \lambda_{n-1})^{k_{n-1}}$, with the deck transformation $(x, y) \mapsto (x, e^{\frac{2 \pi i}{l}} y)$. If we take $\lambda_i$ to be rational and consider the curve given by the same equation, then the deck transformation is not defined over $\mathbf{Q}$, which forces the corresponding field extension to not be normal. Therefore, the first step in the proof is to find the correct condition on $\lambda_i$ so that we can perform the Galois descent in such a way as to make everything, including the deck transformation, defined over $\mathbf{Q}$. \par 
	After this, we will have constructed an etale covering of the parameter space with geometric Galois group either $\SlU(n, q)$ or $\SlL(n, q)$. However, this covering will not be geometrically connected for the same reason that the modular curve $X(N)$ isn't. The Weil pairing is naturally values in roots of unity, so the action of the arithmetic fundamental group over $\mathbf{Q}$ will not preserve the Weil pairing on the nose, only up to scaling. \par 
	We will discover that in some situations, the projective image of the Galois representation is still the simple group $\PSL(n, q)$ or $\PSU(n, q)$. Recall that the image of an element in $\GL(n, q)$ is in $\PSL(n, q)$ if and only if its determinant is an $n$-th power. Using general facts from $p$-adic Hodge theory, we will see that the determinant of the Galois representation under consideration is $\varepsilon ^ {\frac{n}{2}} \psi$, where $\psi$ is a finite order character. Under some congruence conditions, we can have that any $\frac{n}{2}$-th power in $\mathbf{F}_q$ is automatically an $n$-th power, while $\mathrm{gcd}(n , q - 1) \neq 1$. \par 
	From our topological considerations, we see that the determinant of the monodromy action may be an arbitrary $l$-th root of unity. This implies that the "geometric part" of $\psi$ is of order $l$. The remainder of the proof is devoted to the fact that there is no Dirichlet character in the determinant coming from the difference between the arithmetic and geometric monodromy groups. Because this difference is independent of the particular point in the parameter space, it suffices to find a single example of a cyclic covering of curves of degree $l$ with the correct determinant. We give an example of a cyclic quotient of $X_0(l^2)$, whose cohomology is accessible by using modular forms, and find a degeneration of this example where the base becomes a union of rational curves. This allows us to find the required example where the base is $\mathbf{P}^1$, and complete the proof.
	
	\section{Galois descent}
	
	All of the following discussion is contained in \cite[Theorem I.4.11]{malle}, but we translate the statement into algebro-geometric language for the reader's convenience. See also \cite[Chapter 8]{serre}. \par 
	\begin{proposition} \cite[Corollary I.2.7]{malle}
		Let $x_1, \ldots, x_m \in \mathbf{P}^1(\overline{\mathbf{Q}})$ be such that the set $\{x_1, \ldots, x_m \}$ is $G_{\mathbf{Q}}$-invariant. By abuse of notation, for $\sigma \in G_{\mathbf{Q}}$ we will denote $\sigma(x_i)$ by $x_{\sigma(i)}$. Then the action of $G_{\mathbf{Q}}$ on $\pi_1^{ab}(\mathbf{P}^1_{\overline{\mathbf{Q}}} \setminus \{x_1, x_2, \ldots, x_m\})$ is given by $\sigma : \gamma_i \mapsto \gamma_{\sigma(i)}^{\varepsilon(\sigma)}$, where $\gamma_i$ is the small loop around $x_i$, and $\varepsilon : G_{\mathbf{Q}} \to \widehat{\mathbf{Z}} ^ {\times}$ is the cyclotomic character.
	\end{proposition}
	In order to find a cyclic covering of $\mathbf{P}^1$ of prime degree $l$, which descends to $\mathbf{Q}$ together with its deck transformations, it suffices to construct a surjective homomorphism $f : \pi_1^{ab}(\mathbf{P}^1_{\overline{\mathbf{Q}}} \setminus \{x_1, x_2, \ldots, x_m\}) \to \mathbf{Z}/l$, which is invariant under this $G_{\mathbf{Q}}$-action. If we assume, without loss of generality, that $f(\gamma_1)=1$, and if $\sigma_i \in G_{\mathbf{Q}}$ is such that $\sigma_i(1)=i$, then $f(\gamma_i) = f(\gamma_i ^{\varepsilon(\sigma_i) \varepsilon(\sigma_i)^{-1}}) = f(\sigma_i(\gamma_1)^{\varepsilon(\sigma_i)^{-1}}) = \varepsilon(\sigma_i)^{-1}$. The simplest way we can guarantee this to be well-defined is to take $\{x_1, \ldots, x_m \}$ to consist of multiple full orbits of the Galois group $\Gal(\mathbf{Q}(\zeta_l)/\mathbf{Q})$, and define $f$ on each orbit by the formula just derived. \par 
	\begin{remark}
	This discussion can be slightly generalised as follows. For $\mu | l-1$, let $E_{\mu,l}$ be a unique subfield of $\mathbf{Q}(\zeta_l)$ with $[\mathbf{Q}(\zeta_l):E]=\mu$. Then, as above, we can descend a cyclic covering to $E_{\mu, l}$ if we take sets of ramification points in full orbit of $\Gal(\mathbf{Q}(\zeta_l)/E)$. Of course, this implies that the number of branch points must be divisible by $\mu$.
	\end{remark}
	We are now ready to calculate all the dimensions of the Galois representations which occur. Let's say that $\pi : X \to \mathbf{P}^1$ has $r$ ramification points.  \par 
	From the Riemann-Hurwitz formula, we obtain $2g=(l-1) (r-2)$. Observe that the deck tranformation of order $l$ endows the cohomology group $H^1(X; \mathbf{Z})$ with a $\mathbf{Z}[T]/(T^l - 1)$-action in such a way that the $T=1$ subgroup vanishes, because it reduces to $H^1(X;\mathbf{Z})^{T=1} = H^1(X / \langle T \rangle; \mathbf{Z}) = H^1(\mathbf{P}^1; \mathbf{Z}) = 0$. Therefore this action actually factors through $\mathbf{Z}[\zeta_l]$. To get the rank of $H^1(X;\mathbf{Z})$ as an $\mathbf{Z}[\zeta_l]$-module, we have to divide its $\mathbf{Z}$-rank by $l-1$, obtaining $\frac{2g}{l-1} = r - 2$. Therefore, if we reduce this cohomology group modulo a maximal ideal $\mathfrak{m}$ of $\mathbf{Z}[\zeta_l]$ of residue characteristic $p$, we obtain an $(r-2)$-dimensional vector space over $\mathbf{F}_{q}$, $q=p^{\mathrm{ord} (p \mod l)}$. Remebering that $r$ must be divisible by $\mu$, we finally obtain the following congruence conditions which are necessary for our setup to work:
	\begin{theorem} \label{main}
		Let $l$ be an odd prime, let $\mu$ be an integer dividing $l-1$, let $n$ be an integer, such that $n \equiv -2 \mod \mu$, and let $p$ be an odd prime, such that $p \neq l$.  
		\begin{itemize}
			\item If $\frac{l-1}{\mathrm{ord (p \mod l)}}$ is odd, then one can realize a group between  $\PGU(n,q)$ or $\PSU(n,q)$ with $q=p^{\frac{\mathrm{ord} (p \mod l)}{2}}$, as a Galois group over the field $E_{\mu,l}$.
			\item If $\frac{l-1}{\mathrm{ord (p \mod l)}}$ is even, then one can realize a group between  $\PGL(n,q)$ or $\PSL(n,q)$ with $q=p^{\mathrm{ord} (p \mod l)}$, as a Galois group over the field $E_{\mu,l}$.
		\end{itemize}
	\end{theorem}
	\begin{proof}
	Write $n+2 = d \mu$, where $d$ is an integer. Fix a generator $\sigma_0$ of the Galois group $\Gal(\mathbf{Q}(\zeta_l)/E_{\mu, l})$. Let $\mathcal{C}_{\mu,d}$ be the space representing the functor $\mathcal{C}_{\mu,d}(R) = \{ (x_{11}, \ldots, x_{1\mu}, x_{21}, \ldots, x_{2\mu}, \ldots, x_{d\mu}) \in (R \otimes_{\mathbf{Q}} \mathbf{Q}(\zeta_l)), \forall i = 1, \ldots, d, j = 1, \ldots, \mu, \sigma_0(x_{ij}) = x_{i, j+1 \mod \mu}\}$. That means that for every rational point of $\mathcal{C}_{\mu, d}$, the $\mu$-tuple $(x_{i1}, \ldots, x_{i\mu})$ forms a Galois orbit of $\Gal(\mathbf{Q}(\zeta_l)/E_{\mu,l})$. Scheme-theoretically, this is an open subset of the Weil restriction $\mathrm{Res}^{\mathbf{Q}(\zeta_l)}_{E_{\mu,l}} \mathbf{A}^d$, which is a form of the usual parameter space of $d \mu$-tuples of distinct points in $\mathbf{A}^{d \mu}$. What's particularly important for us is that this moduli space is rational. \par 
	To any point of the moduli space $\mathcal{C}_{\mu, d}$, we can uniquely associate a cyclic covering of $\mathbf{P}^1$, ramified at exactly all the points $x_{ij}$, such that both the covering curve and the deck transformation are defined over $\mathbf{Q}$. 
	Let $\mathcal{C}_{\mu, d}(p)$ be the moduli space of these tuples together with a level $p$-structure on the corresponding cyclic covering of $\mathbf{P}^1$.  Let us consider the Galois group $\Gal(E_{\mu,l}(\mathcal{C}_{\mu,d}(p))/E_{\mu,l}(\mathcal{C}_{\mu,d}))$ of the extension generated by the coordiantes of the $p$-torsion points on the Jacobian of the generic cyclic covering. \par 
	Recall from \cite{nesterov} that if $\frac{l-1}{\mathrm{ord (p \mod l)}}$ is odd, then there is a natural nondegenerate sesquilinear form on the $\mathbf{F}_q$-vector space $H^1(X; \mathbf{Z})/\mathfrak{m}$. Futhermore, by the main theorem of the same paper, the topological monodromy group in this case contains $\PSU(n, q)$, and in the other case contains $\PSL(n, q)$. As the parameter space $\mathcal{C}_{\mu,d}$  is rational, we can apply Hilbert irreducibility theorem to conclude that for infinitely many parameter values, the Galois group of the specialization contains either $\PSU(n, q)$ or $\PSL(n, q)$. \par 
	On the other hand, because the Weil pairing is valued in roots of unity and is Galois-equivariant, we can at least guarantee that the Galois action always preserve the Weil pairing up to scaling. It follows that the Galois-action preserves the sesquilinear form obtained from the Weil pairing up to scaling. Therefore, the Galois group of any specialization is always contained in $\PGU(n, q)$ in the unitary case.
	\end{proof}
	 \begin{remark}
	 	Let us make explicit the special cases of the result for small $l$.
	 	\begin{itemize}
	 		\item For $l=3$: If $p=3k+2$ is a prime number, one can realise a group between $\PGU(n,p)$ and $\PSU(n,p)$ as a Galois group over $\mathbf{Q}$ for even $n$, and over $\mathbf{Q}(\omega)$ for odd $n$;  If $p=3k+1$ is a prime number, then one can realise a group between $\PGL(n,p)$ and $\PSL(n,p)$ as a Galois group over $\mathbf{Q}$ for even $n$ and over $\mathbf{Q}(\omega)$ for odd $n$.
	 		\item For $l=5$: If $p=5k \pm 2$ is a prime number, one can realise a group between $\PGU(n,p^2)$ and $\PSU(n,p^2)$ as a Galois group over $\mathbf{Q}$ for $n \equiv 2 \mod 4$, over $\mathbf{Q}(\sqrt{5})$ for $n$ divisible by $4$, and over $\mathbf{Q}(\zeta_5)$ for odd $n$; If $p=5k+1$ is a prime number, then one can realise a group between $\PGL(n,p)$ and $\PSL(n,p)$ as a Galois group over $\mathbf{Q}$ for $n \equiv 2 \mod 4$, over $\mathbf{Q}(\sqrt{5})$ for $n$ divisible by $4$, and over $\mathbf{Q}(\zeta_5)$ for $n$ odd. If $p=5k-1$ is a prime number, then one can realise a group between $\PGL(n,p^2)$ and $\PSL(n,p^2)$ over the same field.
	 		\item For $l=7$: If $p=7k-1$ is a prime number one can realise a group between $\PGU(n,p)$ and $\PSU(n,p)$ as a Galois group over $\mathbf{Q}$ if $n \equiv 4 \mod 6$, over $\mathbf{Q}(\sqrt{-7})$, if $n \equiv 1 \mod 6$, over $\mathbf{Q}(\cos \frac{2 \pi}{7})$, if $n \equiv 3,5 \mod 6$, and over $\mathbf{Q}(\zeta_7)$ otherwise. If $p=7k+3,7k+5$ is a prime number, then one can realise a group between $\PGU(n,p^3)$ and $\PSU(n,p^3)$ over the same field. If $p=7k+1$ is a prime number, then one can realise a group between $\PGL(n,p)$ and $\PSL(n,p)$ as a Galois group over $\mathbf{Q}$ if $n \equiv 4 \mod 6$, over $\mathbf{Q}(\sqrt{-7})$, if $n \equiv 1 \mod 6$, over $\mathbf{Q}(\cos \frac{2 \pi}{7})$, if $n \equiv 3,5 \mod 6$, and over $\mathbf{Q}(\zeta_7)$ otherwise.
	 	\end{itemize}
	 \end{remark}
 	
 	\section{The Multiplier Character and the Determinant}
 	
 	Our next goal is to determine which subgroup between $\PSU(n,q)$ and $\PGU(n,q)$ we actually realize in theorem \ref{main}. The difference between the two groups is completely exhausted by the determinant $\det \rho$ and the multiplier character $\nu$, so this is what we have to compute. Let us specialize for now to the case $\mu = l-1$, $E_{\mu, l} = \mathbf{Q}$.
 	
 	\begin{proposition}
 		For the Galois representations $G_{\mathbf{Q}} \to GU(n,q)$ constructed in the previous theorem, we have $\nu(\rho) = \varepsilon$ and $\det \rho = \varepsilon^{\frac{n}{2}} \psi$ for a character $\psi$ of order dividing $2l$.
 	\end{proposition}
 	\begin{remark}
 		Under the conditions of Theorem 3, $n$ is automatically even.
 	\end{remark}
 	\begin{proof}
 		In this proof, we need to consider the $\mathbf{F}_q$-representation constructed above as a mod $p$ reduction of a $\mathbf{Q}_p(\zeta_l)$-representation. \par
 		We start by determining the Galois action on the sesquilinear form $h$ defining the unitary group. It is related to the cup product in cohomology via $\mathrm{Tr}_{\mathbf{Q}_p(\zeta_l) / \mathbf{Q}_p} h(x,y) = x \cup y$. If we act on both $x$ and $y$ by an element $g \in G_{\mathbf{Q}}$, we get $\mathrm{Tr}_{\mathbf{Q}_p(\zeta_l) / \mathbf{Q}_p} h(gx,gy) = gx \cup gy = \varepsilon(g) x \cup y = \mathrm{Tr}_{\mathbf{Q}_p(\zeta_l) / \mathbf{Q}_p} \varepsilon(g) h(x,y)$. From this we conclude that the multiplier character $\nu(\rho) = \varepsilon$. \par 
 		Recall that for a unitary map, its determinant and multiplier are related by $N_{\mathbf{Q}_p(\zeta_l)/\mathbf{Q}_p(\zeta_l)^+} \det = \nu^n$, where $n$ is the dimension of the vector space in question. We also know that the determinant of our Galois representation is a Hodge-Tate character $G_{\mathbf{Q}} \to \mathbf{Q}_p(\zeta_l)^{\times}$, therefore it must have the form $\varepsilon^k \psi$, where $k$ is an integer and $\psi : G_{\mathbf{Q}} \to \mu_{2l}$ is a finite order character. Substituting $\nu = \varepsilon$, we obtain $N_{\mathbf{Q}_p(\zeta_l)/\mathbf{Q}_p(\zeta_l)^+} \det = \varepsilon^{2k}N_{\mathbf{Q}_p(\zeta_l)/\mathbf{Q}_p(\zeta_l)^+} (\psi) = \varepsilon^n$, hence $k = \frac{n}{2}$. 
 	\end{proof}
 	
 	Similarly, the difference between $\PSL(n, q)$ and $\PGL(n, q)$ is exhausted by the determinant only. 
 	
 	\begin{proposition}
 			For the Galois representations $G_{\mathbf{Q}} \to \GL(n,q)$ constructed in the previous theorem, we have $\det \rho = \varepsilon^{n} \psi$ for a character $\psi$ of order dividing $2l$.
 	\end{proposition}
 	\begin{proof}
 		We note that the eigenvalues of the Frobenius at $p$ acting on $H^1_{et}(X; \mathbf{Q}_p)$ are $p$-Weil numbers of weight $1$. When we compare the eigenvalues of a linear operator on an $\mathbf{Q}_p(\zeta_l)$-vertor space and its underlying $\mathbf{Q}_p$-vector space, the former eigenvalues are included among the latter eigenvalues. Hence, the the eigenvalues of the Frobenius at $p$ on the Galois representation from theorem \ref{main} are also $p$-Weil numbers of weight $1$. The determinant of this Frobenius action is therefore a $p$-Weil number of weight $n$. Hence, $\det \rho = \varepsilon^{n} \psi$ for a finite order character $\psi$.
 	\end{proof}
 	\begin{theorem} \label{unitary}
 		For the projective Galois representations $G_{\mathbf{Q}} \to \PGU(n,q)$ constructed in the previous theorem, if  $l {\not |}\ n$ and if either $q$ is square or $\gcd(n,q+1)=\gcd(\frac{n}{2},q+1)$, then the image is either $\PSU(n,q)$ or $\PSU(n,q)$ has index two in the image.
 	\end{theorem}
 	\begin{proof}
	Recall that there is an isomorphism between $\PGU(n,q)$ and $\PU(n,q)$ constructed as follows. Let $A \in \GU(n,q)$. Because the norm map $\mathbf{F}_{q^2}^{\times} \to \mathbf{F}_{q}^{\times}$ is surjective, we may represent $\nu = \lambda \lambda^{q}$. Then replacing $A$ by $\lambda^{-1}A$, we see that every element in the quotient modulo center is represented by the matrix with $\nu=1$. Moreover, after applying this operation, we get $\det \lambda^{-1}A = \frac{\nu^{\frac{n}{2}}\xi}{\lambda^n} = \Big(\frac{\lambda^{q}}{\lambda} \Big)^{\frac{n}{2}}\xi$ -- a perfect $\frac{n}{2}$-power times an element of $\mu_{2l}$. We can write $\xi = (-1)^s \xi'$, where $\xi' \in \mu_l$, $s = 0, 1$. \par 
	Recall further that the quotient group $\PU(n,q)/\PSU(n,q)$ has order $\mathrm{gcd}(n,q+1)$, and that a unitary matrix $A$ is represented in this quotient by the image of $\det A$ in $\mu_{q+1}/\mu_{q+1}^n = \mu_{q+1}/\mu_{q+1}^{\gcd(n,q+1)}$. If $\xi' \in \mu_l$, and $l {\not |}\ n$, then $\xi'$ always represents a trivial element in this quotient. \par 
	For the Galois representation $\rho$ in question, which have $\nu = \varepsilon$, we see that for any $g \in G_{\mathbf{Q}}$, $\nu(\rho(g)) \in \mathbf{F}_p^{\times}$. If $q$ is a square, then in the contsruction of the isomorphism $\PGU \to \PU$, we can pick $\lambda \in \mathbf{F}_{p^2}$ such that $\lambda^2 = \nu(\rho(g))$. Then $\lambda^q = \lambda$, and $\lambda \lambda^q = \nu(\rho(g))$. In this case, in the expression $\det \lambda^{-1}A = \Big(\frac{\lambda^{q}}{\lambda} \Big)^{\frac{n}{2}}\xi = \xi$, the $\frac{n}{2}$-power part cancels out. Alternatively, if $\gcd(n,q+1)=\gcd(\frac{n}{2},q+1)$, then every $\frac{n}{2}$-power in $\mu_{q+1}$ is automatically an $n$-th power. \par 
	\end{proof}
	Note that because $n$ and $q+1$ are both even, the element $-1$ will always represent a nontrivial element in $\mu_{q+1}/\mu_{q+1}^n$. If we want to arrange that our Galois representation has image exactly $\PSU(n,q)$, it suffices to prove the following:
	\begin{proposition} \label{odd}
		For the Galois representations $G_{\mathbf{Q}} \to \GU(n,q)$ constructed in theorem \ref{main}, when we write $\det \rho = \varepsilon^{\frac{n}{2}} \psi$, the character $\psi$ has odd order.
	\end{proposition}
	\begin{theorem} \label{linear}
		For the projective Galois representations $G_{\mathbf{Q}} \to \PGL(n,q)$ constructed in the previous theorem, if  $l {\not |}\ n$, then the image is either $\PSL(n,q)$ or $\PSL(n,q)$ has index two in the image.
	\end{theorem}
	\begin{proof}
		Recall that the quotient group $\PGL(n,q)/\PSL(n,q)$ has order $\gcd(n,q-1)$, and that the matrix $A$ in represented in this quotient by the image of $\det A$ in $\mathbf{F}_q^{\times}/ \mathbf{F}_q^{\times n} = \mathbf{F}_q^{\times}/ \mathbf{F}_q^{\times \gcd(n,q-1)}$. \par 
		For our Galois representations, we have $\det \rho = \varepsilon^n \psi$, where $\psi$ is of order dividing $2l$. As the $n$-th powers are always trivial in the quotient, and $l {\not |}\ n$, the only part that's possibly nontrivial has order two.
	\end{proof}
	Again, if we want to arrange that the image is exacty $\PSL(n,q)$, we need to prove:
	\begin{proposition} \label{odd'}
		For the Galois representations $G_{\mathbf{Q}} \to \GL(n,q)$ constructed in theorem \ref{main}, when we write $\det \rho = \varepsilon^{\frac{n}{2}} \psi$, the character $\psi$ has odd order.
	\end{proposition}
	\section{Determinant cannot be $-1$}
	In this section, we will prove propositions \ref{odd} and \ref{odd'}. The proof will be uniform in the linear and unitary cases.
	To prove that $\psi$ has odd order, we write $\psi = \delta \psi'$, where $\delta$ has order dividing two, and $\psi'$ has order dividing $l$. First, we make the following observation to show that if $\delta$ is trivial on one curve in the family, then it is trivial for all the others:
	\begin{proposition} \label{indep}
		Let $F$ be a field, let $X$ be a connected algebraic variety over $F$, and let $\mathcal{L}$ be an \'etale local system of finite sets on $X$. If the base change of $\mathcal{L}$ to $X_{\overline{F}}$ is trivial, and for some rational point $x_0 \in X(F)$, the fiber $\mathcal{L}_{x_0}$ has a trivial action of $G_F$, then $\mathcal{L}$ is trivial.
	\end{proposition}
	\begin{proof}
		Indeed, a rational point $x_0$ gives a decomposition $\pi_1^{\et}(X)$ as a semidirect product of $\pi_1^{\et}(X_{\overline{F}})$ and $G_F$, and an \'etale local system of finite sets can be interpreted as a group homomorphism $\pi_1^{\et}(X) \to S_n$. If a group homomorphism is trivial on both $\pi_1^{\et}(X_{\overline{F}})$ and $G_F$, then it is trivial.
	\end{proof}
	Unfortunately, there is no explicit example of a cyclic covering of $\mathbf{P}^1$, where the \'etale cohomology can be easily computed. Thus, we resort to a degeneration technique. Recall 
	\begin{proposition} \cite[9.2.2, example 8]{neron}
		Let $X$ be a semistable curve over a field $F$. Then we have an exact sequence 
		$$0 \to T \to  \mathrm{Pic}^0_X \to \prod_i \mathrm{Pic}^0_{X_i} \to 0$$, where $X_i$ are the normailsations of the irreducible components of $X$, and $T$ is a torus Cartier-dual to the first cohomology group of the graph of the connected components.
	\end{proposition}
	\begin{proposition} \cite[Proposition 2.2.5]{sga}
		Let $R$ be a discrete valuation ring with fraction field $K$ and residue field $k$. Let $A$ be an abelian variety over $K$ with semistable reduction. Then for any prime number $p$ not equal to the characteristic of $k$, the module $T_p(A)/T_p(A)^{I_K}$ is Cartier-dual to $T_p(T')$, where $T'$ is the toric part of the dual abelian variety.
	\end{proposition}
	%add reference for why inertia invariants are H^1 of special fiber
	For Jacobians of curves, we can combine the two propositions using the facts that Jacobians are principally polarized to get:
	\begin{proposition} \label{degen}
		Let $R$ be a discrete valuation ring with fraction field $K$ and residue field $k$. Let $X$ be a curve over $K$ with semistable model $\mathcal{X}$ over $R$. Then for any prime number $p$ not equal to the characteristic of $k$, we have two exact sequences
		$$0 \to T_p(T)(-1) \to H^1(X_{\overline{K}}; \mathbf{Z}_p) \to H^1(\mathcal{X}_{\overline{k}}; \mathbf{Z}_p) \to 0$$
		$$0 \to \bigoplus H^1((X_i)_{\overline{k}}; \mathbf{Z}_p) \to H^1(\mathcal{X}_{\overline{k}}; \mathbf{Z}_p) \to T_p(T)^{\vee} \to 0$$, 
		where $X_i$ are the normalisations of the irreducible components of $\mathcal{X}_k$.
	\end{proposition}
	For any odd prime number $l$, the map $X_1(l^2) \to X_0(l^2)$ is a Galois covering with Galois group $(\mathbf{Z}/l^2 \mathbf{Z})/\{\pm 1\}$. Both curves and all the deck transformations are defined over $\mathbf{Q}$. As the group $(\mathbf{Z}/l^2 \mathbf{Z})/\{\pm 1\}$ has a unique quotient of order $l$, there exists a unique intermediate Galois covering $Y \to X_0(l^2)$ of degree $l$. \par 
	Recall that $X_0(l^2)$ has $l+1$ cusps: two rational cusps corresponding to a cuspidal cubic with marked $\mu_{l^2}$, and a N\'eron $l^2$-gon; and $l-1$ cusps defined over $\mathbf{Q}(\zeta_l)$ coresponding to N\'eron $l$-gons. The covering maps $X_1(l^2) \to X_0(l^2)$ and $Y \to X_0(l^2)$ are ramified precisely at $l-1$ non-rational cusps. The cohomology of $Y$ decomposes into the sum of two-dimensional Galois representations corresponding to various modular forms of weight $2$ level $l^2$. The determinant of these Galois representation is $\varepsilon$ times the Nebentypus character, which is always of odd order. So the analogue of proposition \ref{odd} is true in this example. \par 
	I claim that I can deduce proposition \ref{odd} in general by degenerating the marked curve $(X_0(l^2), \mathrm{non-rational \ cusps})$ to a graph of rational curves in such a way that all the markings are on the same irreducible component.
	\begin{proposition}
		Let $F$ be field of characteristic zero, let $X$ be a smooth connected quasiprojective algebraic variety over $F$, and let $x, y \in X(F)$ be two rational points. Then there exists a smooth irreducible curve on $X$ defined over $F$ and passing through $x$ and $y$.
	\end{proposition}
	\begin{proof}
		If $X$ is a curve, there is nothing to prove. Otherwise, fix a projective embedding of $X \to \mathbf{P}^N$. Consider the space of hyperplanes passing through $x$ and $y$. By \cite[Remark 10.9.2, Exercise 11.3]{heartshorne}, the locus of hyperplanes such that the section is connected and smooth away from $x$ and $y$ is open and nonempty. Because $X$ is smooth, the tangent space $T_x(X)$ is not the whole space $T_x(\mathbf{P}^N)$. Because $\dim X \geq 2$, this tangent space cannot be contained within the line through $x$ and $y$. Hence the locus of hyperplanes transverse to $T_x(X)$ is open and nonempty. Intersecting three open nonempty loci, we obtain a generic set of hyperplanes such that the corresponding section is connected smooth everywhere.
	\end{proof}
	Now we can prove proposition \ref{odd} as follows:
	\begin{proof}
		It suffices to prove that the determinant is of the form $(\text{power of } \varepsilon) \psi$, where $\psi$ is of odd order.
		First, we consider the case $n=l-3$, when the covering of $\mathbf{P}^1$ is ramified at a single orbit. Let $g$ be the genus of $X_0(l^2)$. In the twisted moduli space $\overline{\mathcal{M}}$ representing the functor $\overline{\mathcal{M}} (R) = \{(X, x_1, \ldots, x_{l-1}), X \text{ is a stable genus } g \text{ curve over } R \otimes_{\mathbf{Q}} \mathbf{Q}(\zeta_l), \ x_i \in X(R \otimes_{\mathbf{Q}} \mathbf{Q}(\zeta_l)), \forall i, \sigma_0(x_i) = x_{i+1} \}$, we may find a curve $C$ passing through two points:
		\begin{itemize}
			\item the point $x=(X_0(l^2), \mathrm{non-rational \ cusps})$;
			\item a point  $y$ corresponding to a rational curve with $g$ nodes, with all the markings forming a unique $\Gal(\mathbf{Q}(\zeta_l)/\mathbf{Q})$-orbit.
		\end{itemize}
	In proposition \ref{degen}, take $R$ to be the completed local ring $\widehat{\mathcal{O}}_{C,y}$. Denote by $\eta$ the generic point of $C$, and denote by $X_{\eta}$, $X_y$ the normalisations of cyclic coverings of the corresponding curves branched at exactly the marked points. The absolute Galois group of the residue field $\mathbf{Q}((t))$ of $\eta$ contains a canonical copy of $\GalQ$, which we identify with its image. By proposition \ref{degen}, for any $\sigma \in \GalQ$, $\det (\sigma | H^1(X_{\eta, \overline{\mathbf{Q}((t))}}; \mathbf{Z}_p)) = \varepsilon^{-r}(\sigma) \det (\sigma | H^1(X_{y, \overline{\mathbf{Q}}}); \mathbf{Z}_p)$, where $r$ is the number of nodes of $X_y$. Applying the same reasoning for the completed local ring at $x$, we get $\det(\sigma | H^1(X_{x, \Qbar}; \mathbf{Z}_p))=  \varepsilon^{-r}(\sigma) \det(\sigma|H^1(X_{y, \Qbar}; \mathbf{Z}_p))$. Therefore, $X_y$ is an example of a cyclic covering of $\mathbf{P}^1$ where the determinant of our Galois representation has the form $(\text{power of }\varepsilon) \times \psi$, where $\psi$ is of odd order. Hence, by proposition \ref{indep}, for all the cyclic coverings of $\mathbf{P}^1$ ramified at a single orbit, the determinant has the form $(\text{power of } \varepsilon)\times \psi$, where $\psi$ is of odd order.\par 
	Now let $Y$ be a cyclic covering of $\mathbf{P}^1$ branched at more than one orbit. Find a curve $C$ in the moduli space of curves of genus $0$ with marked $\Gal(\mathbf{Q}(\zeta_l)/\mathbf{Q})$-orbits passing through two points:
	\begin{itemize}
		\item the point $(\mathbf{P}^1, \text{branch locus of } Y \to \mathbf{P}^1)$;
		\item A chain of $\mathbf{P}^1$'s with markings on any component forming a single orbit.
	\end{itemize}
	Let $Y_i$ be cyclic coverings of the components of the latter curve branched at the markings. By applying the same reasoning as before, we get that $\det(\sigma | H^1(Y, \mathbf{Z}_p)) = \varepsilon^{-r}(\sigma) \prod \det(\sigma | H^1(Y_i, \mathbf{Z}_p))$. From the previous case, we know the right hand side to be a power of $\varepsilon$ times the character of odd order, hence the left hand side is of the same form.
	\end{proof}
	\begin{proof} (of theorem \ref{final})
		Let $l$ be an odd prime, let $n$ be a natural number, not congruent to $-1$ mod $l$. Let $e = \mathrm{ord}\ p \mod l$. Consider the Galois representations constructed in theorem \ref{main} with $\mu = l-1$ and $n = (l-1)n - 2$. If $\frac{l-1}{e}$ is even, it follows from the propositions \ref{linear} and \ref{odd'} that these Galois representations have projective image $\PSL(n,p^e)$. \par 
		If $\frac{l-1}{e}$ is odd, we sometimes have to verify extra congruence conditions. If $e$ is a multiple of $4$, then $p^{\frac{e}{2}}$ is a square, and \ref{unitary} applies. Otherwise, let $r = v_2(p+1)$, and assume that $\frac{l-1}{2} n \equiv 1 \mod 2^r$. I claim that for $q=p^{\frac{e}{2}}$, we have $\gcd((l-1)n-2,q+1)=\gcd(\frac{(l-1)n-2}{2},q+1)$, so that \ref{unitary} applies. Indeed, the only case when $\gcd((l-1)n-2,q+1)>\gcd(\frac{(l-1)n-2}{2},q+1)$ is when $v_2(q+1) > v_2(\frac{(l-1)n-2}{2})= v_2(\frac{l-1}{2}n - 1)$. I claim that $v_2(q+1) = v_2(p+1)$, so this inequality does not happen. \par 
		To prove this claim, recall that we are in the case where $\frac{e}{2}$ is odd. By assumtion, $p + 1 \equiv 2^r \mod 2^{r+1}$. From this, we obtain $p^{\frac{e}{2}} \equiv (2^r - 1)^{\frac{e}{2}} \equiv (-1)^{\frac{e}{2}} + \frac{e}{2} 2^r \equiv -1 + 2^r \mod 2^{r+1}$, so $p^{\frac{e}{2}} + 1$ is also not divisible by $2^{r+1}$.
	\end{proof}

\end{document}